\documentclass[a4paper, 10pt, parskip=half]{scrartcl}

\usepackage[utf8]{inputenc}
\usepackage[T1]{fontenc}
\usepackage{lmodern}
\usepackage{amsmath}
\usepackage{amssymb}
\usepackage{amsthm}
\usepackage{amsfonts}
\usepackage{dsfont}
\usepackage{mathtools}
\usepackage{bbm}
\usepackage{marginnote}
\usepackage{hyperref}
\hypersetup{%
  colorlinks=true,
  linkcolor=black,
  citecolor=black,
  urlcolor=blue,
  pdftitle={On the optimality of upper estimates near blow-up in quasilinear Keller–Segel systems},
  pdfauthor={Mario Fuest},
  pdfkeywords={blow-up profile, nonlinear diffusion, gradient estimates, chemotaxis},
  bookmarksopen=true,
}

\RequirePackage{geometry}
\newcommand{\R}{\mathbb{R}}

\newcommand{\N}{\mathbb{N}}

\newcommand{\mc}[1]{\mathcal{#1}}

\newcommand{\ur}[1]{\mathrm{#1}}
\newcommand{\ure}{\ur e}

\newcommand{\eps}{\varepsilon}

\newcommand{\gt}{>}
\newcommand{\lt}{<}

\DeclareMathOperator{\id}{id}

\newcommand{\defs}{\coloneqq}
\newcommand{\sfed}{\eqqcolon}

\newcommand{\ra}{\rightarrow}

\newcommand{\sea}{\searrow}

\newcommand{\ol}{\overline}
\newcommand{\ul}{\underline}

\newcommand{\ds}{\,\mathrm{d}s}

\newcommand{\dr}{\,\mathrm{d}r}

\newcommand{\dsigma}{\,\mathrm{d}\sigma}

\newcommand{\drho}{\,\mathrm{d}\rho}

\DeclareMathOperator{\sign}{sign}

\newcommand{\embed}{\hookrightarrow}

\newcommand{\hp}{\hphantom}
\newcommand{\pe}{\mathrel{\hp{=}}}

\newcommand{\tmax}{T_{\max}}

\newcommand{\intom}{\int_\Omega}

\newcommand{\Ombar}{\ol \Omega}

\newcommand{\spaceused}{\Omega}
\newcommand{\mainspace}[1]{\renewcommand{\spaceused}{#1}}
\newcommand{\leb}[2][\spaceused]{\ensuremath{L^{#2}(#1)}}
\newcommand{\sob}[3][\spaceused]{\ensuremath{W^{#2, #3}(#1)}}
\newcommand{\sobn}[3][\spaceused]{\ensuremath{W_N^{#2, #3}(#1)}}
\newcommand{\con}[2][\ol \spaceused]{\ensuremath{C^{#2}(#1)}}

\newcommand{\pu}{\mathbbmss p}
\newcommand{\qu}{\mathbbm q}

\newcommand{\tops}{\texorpdfstring}

\renewcommand{\paragraph}[2][.]{\textbf {#2#1}}

\makeatletter
\renewenvironment{proof}[1][\proofname]{\par
  \pushQED{\qed}%
  \normalfont \topsep0\p@\relax
  \trivlist
  \item[\hskip\labelsep\scshape
  #1\@addpunct{.}]\ignorespaces
}{%
  \popQED\endtrivlist\@endpefalse
}
\makeatother

\newtheorem{base}{Base}[section]
\numberwithin{equation}{section}

\theoremstyle{plain}
\newtheorem{theorem}[base]{Theorem} \newtheorem*{theorem*}{Theroem}
\newtheorem{lemma}[base]{Lemma} \newtheorem*{lemma*}{Lemma}
\newtheorem{prop}[base]{Proposition} \newtheorem*{prop*}{Proposition}
 \newtheorem*{cor*}{Corollary}

\theoremstyle{definition}
 \newtheorem*{definition*}{Definition}
 \newtheorem*{example*}{Example}
 \newtheorem*{cond*}{Condition}

\newtheorem{remark}[base]{Remark} \newtheorem*{remark*}{Remark}

\newcounter{localconst}[base]
\makeatletter
\newcommand{\newlc}[2][]{%
\refstepcounter{localconst}%
\ltx@label{lc:\thesection:\arabic{base}:#2}%
c_{\ref*{lc:\thesection:\arabic{base}:#2}}}
\makeatother
\newcommand{\lc}[2][]{c_{\ref*{lc:\thesection:\arabic{base}:#2}}}

\begin{document}
\setkomafont{title}{\normalfont \Large}
\title{On the optimality of upper estimates near blow-up in quasilinear Keller–Segel systems}
\author{Mario Fuest\footnote{fuestm@math.uni-paderborn.de}\\
{\small Institut f\"ur Mathematik, Universit\"at Paderborn,}\\
{\small 33098 Paderborn, Germany}
}
\date{}

\maketitle

\KOMAoptions{abstract=true}
\begin{abstract}
  \noindent
  Solutions $(u, v)$ to the chemotaxis system
  \begin{align*}
    \begin{cases}
      u_t = \nabla \cdot ( (u+1)^{m-1} \nabla u - u (u+1)^{q-1} \nabla v), \\
      \tau v_t = \Delta v - v + u
    \end{cases}
  \end{align*}
  in a ball $\Omega \subset \mathbb R^n$, $n \ge 2$,
  wherein $m, q \in \mathbb R$ and $\tau \in \{0, 1\}$ are given parameters with $m - q > -1$,
  cannot blow up in finite time provided $u$ is uniformly-in-time bounded in $L^p(\Omega)$
  for some $p > p_0 \defs \frac n2 (1 - (m - q))$.\\[0.5pt]
  For radially symmetric solutions,
  we show that, if $u$ is only bounded in $L^{p_0}(\Omega)$ and the technical condition $m > \frac{n-2 p_0}{n}$ is fulfilled,
  then, for any $\alpha > \frac{n}{p_0}$, there is $C > 0$ with
  \begin{align*}
    u(x, t) \leq C |x|^{-\alpha} \qquad \text{for all $x \in \Omega$ and $t \in (0, T_{\max})$},
  \end{align*}
  $T_{\max} \in (0, \infty]$ denoting the maximal existence time.
  This is essentially optimal in the sense that, if this estimate held for any $\alpha < \frac{n}{p_0}$,
  then $u$ would already be bounded in $L^{p}(\Omega)$ for some $p > p_0$.
  \\[0.5pt]
  Moreover, we also give certain upper estimates for chemotaxis systems with nonlinear signal production,
  even without any additional boundedness assumptions on $u$.\\[0.5pt]
  The proof is mainly based on deriving pointwise gradient estimates for solutions of 
  the Poisson or heat equation with a source term uniformly-in-time bounded in $L^{p_0}(\Omega)$.\\[0.5pt]
  \textbf{Key words:} {blow-up profile, nonlinear diffusion, gradient estimates, chemotaxis}\\
  \textbf{AMS Classification (2020):} {35B40 (primary); 35K40, 35K65, 92C17 (secondary)}
\end{abstract}

\section{Introduction}
In the first and main part of the present article, we establish pointwise upper gradient estimates for solutions to 
\begin{align}\label{prob:v}
  \begin{cases}
    \tau v_t = \Delta v - v + g               & \text{in $\Omega \times (0, T)$}, \\
    \partial_\nu v = 0                        & \text{on $\partial \Omega \times (0, T)$}, \\
    v(\cdot, 0) = v_0 \text{ if $\tau \gt 0$} & \text{in $\Omega$},
  \end{cases}
\end{align}
where $\Omega = B_R(0)$, $R \gt 0$, is an $n$-dimensional ball,
$\tau \ge 0$, $T \in (0, \infty)$
and $v_0$ and $g$ are sufficiently smooth given functions on $\Omega$ and $\Omega \times (0, T)$, respectively.
Elliptic or parabolic regularity theory (cf.\ Lemma~\ref{lm:ell_v_sob2p} and Lemma~\ref{lm:v_sob1p} below) and embedding theorems
warrant that, if $g$ is uniformly-in-time bounded $\leb\qu$ for some $\qu \in [1, n]$,
then $v$ is uniformly-in-time bounded in $\sob1p$ for all $p \in [1, \frac{n\qu}{n-\qu})$.

An estimate of the form
\begin{align}\label{eq:intro:nabla_v_pw}
  |\nabla v(x, t)| \le C_\beta |x|^{-\beta}
  \qquad \text{for all $x \in \Omega$ and $t \in (0, T)$}
\end{align}
for some $\beta \lt \frac{n-\qu}{\qu}$ would imply
\begin{align*}
      \sup_{t \in (0, T)} \intom |\nabla v(\cdot, t)|^p
  \le C_\beta^p \omega_{n-1} \int_0^R r^{n-1 - p \beta} \dr
  \lt \infty
\end{align*}
for all $p \in (0, \frac{n}{\beta})$ and hence in particular for $p = \frac{\frac{n}{\beta} + \frac{n \qu}{n-\qu}}{2} \gt \frac{n\qu}{n-\qu}$.
Thus, assuming that the uniform-in-time bounds discussed above are optimal,
such an estimate should not be obtainable if one only requires $\sup_{t \in (0, T)} \|g(\cdot, t)\|_{\leb\qu}$ to be finite.
However, we achieve \eqref{eq:intro:nabla_v_pw} for all $\beta \gt \frac{n-\qu}{\qu}$.
We conjecture that this estimate, possibly up to equality therein, is optimal.

In the elliptic case, the corresponding proof is quite short:
In Section~\ref{sec:elliptic},
we first derive an $L^\qu$~bound for $\Delta v$ and then make use of the symmetry assumption to obtain
\begin{prop}\label{prop:main_e}
  Let $n \ge 2$, $R \gt 0$, $\Omega \defs B_R(0) \subset \R^n$,
  $M \gt 0$, $\qu \in [1, n]$ and $\beta \ge \frac{n-\qu}{\qu}$.
  There is $C \gt 0$ such that whenever
  $g \in \con0$ is a radially symmetric function fulfilling
  \begin{align}\label{eq:cond_g_e}
    \|g\|_{\leb{\qu}} \le M
  \end{align}
  and $v \in \con2$ solves
  \begin{align}\label{prob:v_e}
    \begin{cases}
      0 = \Delta v - v + g & \text{in $\Omega$}, \\
      \partial_\nu v = 0   & \text{on $\partial \Omega$},
    \end{cases}
  \end{align}
  then
  \begin{align}\label{eq:main_e:statement}
    |\nabla v(x)| \le C |x|^{-\beta} \qquad \text{for all $x \in \Ombar$.}
  \end{align} 
\end{prop}
 
In principle, one could argue similarly in the parabolic setting,
although one would at least need to require $v_0 \in \sob2\qu$ with $\partial_\nu v_0 = 0$ on $\partial \Omega$ in the sense of traces---%
or $v$ cannot be uniformly-in-time bounded in $\sob2\qu$.
Not wanting to impose such an unnatural requirement, we argue differently and rely on various semigroup estimates, 
which are introduced in Section~\ref{sec:semigroup}, instead.

For $\qu \in (1, \frac n2]$, we can follow \cite[Section~3]{WinklerBlowupProfilesLife}, where corresponding estimates have been derived for $\qu = 1$.
The main idea is to notice that $z \defs \zeta^\beta v$, where $\zeta(x) \approx |x|$, solves a certain initial boundary value problem
and then make use of several semigroup estimates to obtain an $L^\infty$ bound for $\nabla z$---%
which in turn together with pointwise upper bounds for $v$ (cf.\ Lemma~\ref{lm:v_upper_est}) implies~\eqref{eq:intro:nabla_v_pw}.

However, these arguments rely in several places on the fact
that $\qu \in (1, \frac n2]$ and $\beta \gt \frac{n-\qu}{\qu}$ imply $\beta \gt 1$ and hence $\zeta^\beta \in \con1$.
Switching to radial notation, this for instance means that $z_r(0, \cdot) \equiv 0$.
For $\qu \in (\frac n2, n]$ and thus possibly $\beta \in (0, 1)$, this is no longer the case.
We overcome this problem by considering (for $\qu \in (\frac n2, n]$)
\begin{align}\label{eq:intro:z}
  z(x, t) \defs \zeta^\beta(x) (v(x, t) - v(0, t)), \qquad (x, t) \in \Ombar \times [0, T),
\end{align}
instead.
Due to uniform-in-time Hölder bounds (see Lemma~\ref{lm:v_hoelder}),
we then obtain $z_r(0, \cdot) \equiv 0$
and an $L^\infty$~bound for $\nabla z$ again implies \eqref{eq:intro:nabla_v_pw}.
On the other hand, compared to $\zeta^\beta v$, a new problem arises for $z$ defined as in \eqref{eq:intro:z}:
The time derivative of $z$ now additionally includes $\zeta^\beta v_t(0, \cdot)$.
In order to handle this term, we first derive time Hölder bounds for $v$ in Lemma~\ref{lm:v_time_hoelder}
and then apply more subtle semigroup arguments as in the case of $\qu \in (1, \frac n2]$ in Lemma~\ref{lm:nabla_v_pw}.

Finally, we arrive at
\begin{theorem}\label{th:main_p}
  Let $n \ge 2$, $R \gt 0$, $\Omega \defs B_R(0) \subset \R^n$.
  For every $M \gt 0$, $\qu \in (1, n]$, $\beta \gt \frac{n-\qu}{\qu}$ and $p_0 \gt \max\{\frac{n}{\beta}, 1\}$,
  there is $C \gt 0$ with the following property:
  Suppose $\tau \gt 0$, $T \in (0, \infty]$ and that
  \begin{align}\label{eq:cond_v0_p}
    v_0 \in \con0 \text{ is radially symmetric and nonnegative with }
    \|v_0\|_{\sob1{p_0}} + \||x|^{\beta} \nabla v_0\|_{\leb\infty} \le M
  \end{align}
  as well as
  \begin{align}\label{eq:cond_g_p}
    g \in C^0(\Ombar \times [0, T)) \text{ is radially symmetric with }
    \sup_{t \in (0, T)} \|g(\cdot, t)\|_{\leb\qu} \le M.
  \end{align}
  Then
  \begin{align}\label{eq:main_p:statement}
    |\nabla v(x, t)| \le C |x|^{-\beta} \qquad \text{for all $x \in \Ombar$ and $t \in [0, T)$},
  \end{align} 
  provided $v \in C^0(\Ombar \times [0, T)) \cap C^{2, 1}(\Ombar \times (0, T))$ is a nonnegative classical solution of
  \begin{align}\label{prob:v_p}
    \begin{cases}
      \tau v_t = \Delta v - v + g & \text{in $\Omega \times (0, T)$}, \\
      \partial_\nu v = 0          & \text{on $\partial \Omega \times (0, T)$}, \\
      v(\cdot, 0) = v_0           & \text{in $\Omega$}.
    \end{cases}
  \end{align}
\end{theorem}

\begin{remark}
  \begin{enumerate}
    \item[(i)]
      In \cite[Lemma~3.4]{WinklerBlowupProfilesLife}, corresponding estimates have been derived for $\tau = 1$ and $\qu = 1$
      (provided that in addition to $\eqref{eq:cond_g_p}$, certain pointwise upper estimates of $|g|$ are known).
      This is the reason why we concern ourselves only with $\qu \gt 1$ in Theorem~\ref{th:main_p}.

    \item[(ii)]
      The constant $C$ in Theorem~\ref{th:main_p} evidently needs at least to depend on $\||x|^{\beta} \nabla v_0\|_{\leb\infty}$
      and we avoid further dependencies on the initial data as much as possible;
      in particular, we do neither rely on a $\sob2\qu$ bound nor on fulfillment of certain boundary conditions.
      For technical reasons, however, we need to require \eqref{eq:cond_v0_p},
      which is nearly optimal in the sense that a bound of $\||x|^{\beta} \nabla v_0\|_{\leb\infty}$
      implies bounds for $\|\nabla v_0\|_{\leb p}$ for all $p \in [1, \frac{n}{\beta})$.
  \end{enumerate}
\end{remark}

Next, we apply Proposition~\ref{prop:main_e} and Theorem~\ref{th:main_p}
to the solutions (or, more precisely, to their second components) of the quasilinear chemotaxis system
\begin{align}\label{prob:ks}
  \begin{cases}
    u_t = \nabla \cdot (D(u, v) \nabla u - S(u, v) \nabla v),              & \text{in $\Omega \times (0, T)$}, \\
    \tau v_t = \Delta v - v + f(u, v),                                     & \text{in $\Omega \times (0, T)$}, \\
    (D(u, v) \nabla u - S(u, v) \nabla v) \cdot \nu = \partial_\nu v  = 0, & \text{on $\partial \Omega \times (0, T)$}, \\
    u(\cdot, 0) = u_0                                                      & \text{in $\Omega$}, \\
    v(\cdot, 0) = v_0 \text{ if $\tau \gt 0$},                             & \text{in $\Omega$},
  \end{cases}
\end{align}
where again $\Omega$ is an $n$-dimensional ball, $\tau \ge 0$, $T \in (0, \infty]$ and $u_0, v_0, D, S, f$ are given functions.
Such systems aim to describe chemotaxis, the partially directed movement of organisms $u$ towards a chemical stimulus $v$
and have (for certain choices of parameters) first been proposed by Keller and Segel \cite{KellerSegelTravelingBandsChemotactic1971}.
In certain biological settings, the functions $D$ and $S$ need to be nonlinear---%
accounting for volume-filling effects~%
\cite{HillenPainterUserGuidePDE2009, PainterHillenVolumefillingQuorumsensingModels2002, WrzosekVolumeFillingEffect2010},
immotility of the attracted organisms~\cite{FuEtAlStripeFormationBacterial2012, LeyvaEtAlEffectsNutrientChemotaxis2013}
or saturation of the chemotactic sensitivity~\cite{KalininEtAlLogarithmicSensingEscherichia2009}, for instance.
For a broader overview on chemotaxis systems, we refer to the survey \cite{BellomoEtAlMathematicalTheoryKeller2015}.
 
Before stating our new findings, let us briefly discuss some known results.
For the sake of exposition,
we confine ourselves mainly with the prototypical choices $D(u, v) = (u + 1)^{m-1}$, $S(u, v) = u (u + 1)^{q-1}$ and $f(u, v) = u$,
where $m, q \in \R$ are given parameters,

For the question whether solutions to \eqref{prob:ks} always exist globally,
the value $\frac{n-2}{n}$, $n$ denoting the space dimension, distinguishes between boundedness and blow-up in either finite or infinite time:
If $m-q \gt \frac{n-2}{n}$, solutions to \eqref{prob:ks} remain bounded and hence exist globally
while for $m-q \lt \frac{n-2}{n}$, in multi-dimensional balls, there are initial data leading to unbounded solutions 
(cf.\ \cite{LankeitInfiniteTimeBlowup2020, WinklerDjieBoundednessFinitetimeCollapse2010} for the parabolic--elliptic and
\cite{HorstmannWinklerBoundednessVsBlowup2005, IshidaEtAlBoundednessQuasilinearKeller2014, TaoWinklerBoundednessQuasilinearParabolic2012}
for the parabolic--parabolic case as well as for instance
\cite{CalvezCarrilloVolumeEffectsKeller2006, CieslakWinklerFinitetimeBlowupQuasilinear2008, KowalczykSzymanskaGlobalExistenceSolutions2008, SenbaSuzukiQuasilinearParabolicSystem2006} for earlier partial and related results in this direction).
Similar results are also available for functions $D$ and $S$ decaying exponentially fast in $u$
(see \cite{CieslakWinklerGlobalBoundedSolutions2017} for boundedness in 2D,
\cite{WinklerDoesVolumefillingEffect2009} for the existence of unbounded solutions
and \cite{WinklerGlobalExistenceSlow2017} for the possibility of infinite-time blow-up,
for instance).

In the parabolic--elliptic setting, the sign of $q$ determines whether finite-time blow-up is possible.
That is, while for $q \le 0$ and arbitrary $m \in \R$, solutions to \eqref{prob:ks} are always global in time
and hence unbounded ones have to blow up in infinite time~\cite{LankeitInfiniteTimeBlowup2020},
finite-time blow-up has been detected in the radially symmetric setting for ($m-q \lt \frac{n-2}{n}$ and) $q \gt 0$ in a slightly simplified system~%
\cite{WinklerDjieBoundednessFinitetimeCollapse2010}.
For the fully parabolic case, the situation is similar but not yet as conclusive.
Again, solutions are always global in time for $q \le 0$~\cite{WinklerGlobalClassicalSolvability2019}
but, to the best of our knowledge, finite-time blow-up is only known to occur in multi-dimensional balls if $m-q \lt \frac{n-2}{n}$ and
either $m \ge 1$ (and hence $q \gt \frac2n \gt 0$) or $m \in \R$ and $q \ge 1$~%
\cite{CieslakStinnerFinitetimeBlowupGlobalintime2012, CieslakStinnerFiniteTimeBlowupSupercritical2014, CieslakStinnerNewCriticalExponents2015}.
(For the one-dimensional case, see \cite{CieslakLaurencotFiniteTimeBlowup2010}.)
However, it has been conjectured (for instance in \cite{WinklerGlobalClassicalSolvability2019}) that solutions blowing up in finite time
also exist for the remaining cases ($m-q \lt \frac{n-2}{n}$ and) $m \lt 1$ or $q \in (0, 1)$.

Regarding the behavior of solutions blowing up in finite time near their blow-up time,
some partial results are available for the special case $m=q=1$.
The probably most striking result in this direction is the occurrence of chemotactic collapse; that is,
solutions in two-dimensional balls may converge to a Dirac-type distribution,
both in the parabolic--elliptic~\cite{SenbaSuzukiChemotacticCollapseParabolicelliptic2001}
and in the parabolic--parabolic~\cite{HerreroVelazquezBlowupMechanismChemotaxis1997, NagaiEtAlChemotacticCollapseParabolic2000}
setting.

Moreover, in the radially symmetric multi-dimensional setting,
there are solutions $(u, v)$ blowing up in finite-time which converge pointwise (in $\Omega \setminus \{0\}$) to so called blow-up profiles $(U, V)$,
which for every 
\begin{align*}
  \begin{cases}
    \alpha \ge 2,       & \tau = 0, \\
    \alpha \gt n(n-1),  & \tau = 1
  \end{cases}
  \quad \text{and} \quad
  \beta \gt n-1
\end{align*}
fulfill
\begin{align*}
  U(x) \le C |x|^{-\alpha}
  \quad \text{and} \quad
  V(x) \le C |x|^{-\beta}
  \qquad \text{for all $x \in \Omega$}
\end{align*}
for some $C \gt 0$
(see \cite{SoupletWinklerBlowupProfilesParabolicelliptic2018} for the parabolic--elliptic
and \cite{WinklerBlowupProfilesLife} for the parabolic--parabolic case).

Recently, these results have been extended to quasilinear Keller--Segel systems~\cite{FuestBlowupProfilesQuasilinear2020}:
Again in $n$-dimensional balls, $n \ge 2$, but for arbitrary $m \gt \frac{n-2}{n}$, $m-q \in (-\frac1n, \frac{n-2}{n}]$,
$\alpha \gt \frac{n(n-1)}{(m-q)n + 1}$ and $\beta \gt n-1$,
solutions $(u, v)$ of \eqref{prob:ks} blowing up at $\tmax \in (0, \infty)$ fulfill
\begin{align*}
  u(x, t) \le C |x|^{-\alpha}
  \quad \text{and} \quad
  v(x, t) \le C |x|^{-\beta}
  \qquad \text{for all $x \in \Omega$ and $t \in (0, \tmax)$}
\end{align*}
for some $C \gt 0$.
Apart from certain corner cases, however,
it is to the best of our knowledge not known whether the exponents $\alpha$ and $\beta$ therein are (essentially) optimal.

In the present article, we apply Proposition~\ref{prop:main_e} and Theorem~\ref{th:main_p} in order to improve on these estimates---%
provided that the first solution component is uniformly-in-time bounded in $\leb\pu$ for some $\pu \gt 1$.

\begin{theorem}\label{th:pw_ks}
  Let $n \ge 2$, $R \gt 0$, $\Omega \defs B_R(0) \subset \R^n$ and
  \begin{align*}
    m, q \in \R,
    s \gt 0,
    \tau \ge 0,
    K_{D,1}, K_{D,2}, K_S, K_f \gt 0, M \gt 0, \pu \in [\max\{s, 1\}, ns]
  \end{align*}
  be such that
  \begin{align}\label{eq:pw_ks:cond_m_q}
    m-q \in \left(-\frac{\pu}{n}, \frac{ns-2\pu}{n} \right]
    \quad \text{and} \quad
    m \gt \frac{n-2\pu}{n}.
  \end{align}
  For any
  \begin{align}\label{eq:pw_ks:cond_alpha}
    \alpha \gt \ul \alpha \defs \frac{n (ns-\pu)}{[(m-q)n + \pu] \pu}
    \quad \text{and} \quad
    \beta \gt \frac{ns-\pu}{\pu},
  \end{align}
  we can find $C \gt 0$ such that whenever
  $(u, v) \in \left( C^0(\Ombar \times [0, T)) \cap C^{2, 1}(\Ombar \times (0, T))\right)^2$, $T \in (0, \infty]$,
  with
  \begin{align}\label{eq:pw_ks:cond_M}
    \sup_{t \in (0, T)} \|u(\cdot, t)\|_{\leb{\pu}} \le M
  \end{align}
  is a nonnegative, radially symmetric solution of \eqref{prob:ks}, where
  \begin{align*}
    D, S \in C^1([0, \infty)^2), \quad
    f \in C^0([0, \infty)^2), \quad
    0 \le u_0 \in \con0
    \quad \text{and} \quad
    0 \le v_0 \in \con0
  \end{align*}
  fulfill
  \begin{align*} 
      &\inf_{\sigma \ge 0} D(\rho, \sigma) \ge K_{D,1} \rho^{m-1}, \\
      &\sup_{\sigma \ge 0} D(\rho, \sigma) \le K_{D,2} \max\{\rho, 1\}^{m-1} \\
      &\sup_{\sigma \ge 0} |S(\rho, \sigma)| \le K_S \max\{\rho, 1\}^q \quad \text{and} \\
      &\sup_{\sigma \ge 0} |f(\rho, \sigma)| \le K_f \max\{\rho, 1\}^s
  \end{align*}
  for all $\rho \ge 0$ as well as
  \begin{align*}
    u_0(x) \le M |x|^{-\alpha} \text{ for all $x \in \Omega$}
    \quad \text{and} \quad
    \|v_0\|_{\sob1\infty} \le M,
  \end{align*}
  then
  \begin{align}\label{eq:pw_ks:statement}
    u(x, t) \le C |x|^{-\alpha}
    \quad \text{and} \quad
    |\nabla v(x, t)| \le C |x|^{-\beta}
    \qquad \text{for all $x \in \Omega$ and $t \in (0, T)$}.
  \end{align}
\end{theorem}

As a first application of Theorem~\ref{th:pw_ks}, let us state
\begin{remark}
  To the best of our knowledge, the results above give the first estimates of type~\eqref{eq:pw_ks:statement}
  for chemotaxis systems with nonlinear signal production.
  For instance, letting $u_0 \in \con0$, $v_0 \in \sob1\infty$, $m = q = 1$, $\tau \ge 0$, $\pu = 1$, $s \in (\frac2n, 1]$
  and $\eps \gt 0$,
  solutions of
  \begin{align*}
    \begin{cases}
      u_t = \Delta u - \nabla \cdot (u \nabla v),  & \text{in $\Omega \times (0, T)$}, \\
      \tau v_t = \Delta v - v + u^s,               & \text{in $\Omega \times (0, T)$}, \\
      \partial_\nu u = \partial_\nu v = 0,         & \text{on $\partial \Omega \times (0, T)$}, \\
      u(\cdot, 0) = u_0,                           & \text{in $\Omega$}, \\
      v(\cdot, 0) = v_0 \text{ if $\tau \gt 0$},   & \text{in $\Omega$}
    \end{cases}
  \end{align*}
  fulfill
  \begin{align*}
    u(x, t) \le C |x|^{-n (ns - 1) - \eps}
    \qquad \text{for all $x \in \Omega$ and $t \in (0, T)$}
  \end{align*}
  for some $C \gt 0$.
\end{remark}

Next, we show that Theorem~\ref{th:pw_ks} implies a certain (essentially) conditional optimality for pointwise upper estimates
of solutions to \eqref{prob:ks}.
\begin{remark}
  Suppose $s = 1$ and
  \begin{align}
    m - q \in \left(-1, \frac{n-2}{n} \right]
    \quad \text{as well as} \quad
    q \gt 0
  \end{align}
  and that \eqref{eq:pw_ks:cond_M} holds for
  \begin{align}
    \pu = \frac n2(1-(m-q)) \in [1, ns).
  \end{align}
  Then
  \begin{align*}
    m - q = \frac{n - 2\pu}{n} \in \left(-\frac{\pu}{n}, \frac{n - 2\pu}{n} \right],
  \end{align*}
  hence \eqref{eq:pw_ks:cond_m_q} is fulfilled.
  This implies that for $\ul \alpha$ \eqref{eq:pw_ks:cond_alpha}, we have
  \begin{align*}
      \ul \alpha
    = \frac{n}{\pu} \cdot \frac{n-\pu}{(m-q)n + \pu}
    = \frac{n}{\pu} \cdot \frac{\frac n2 + \frac{(m-q)n}{2}}{\frac n2 + \frac{(m-q)n}{2}}
    = \frac{n}{\pu}
    = \frac{2}{1-(m-q)}
  \end{align*}
  so that \cite[Corollary~2.3]{FreitagBlowupProfilesRefined2018} asserts that condition \eqref{eq:pw_ks:cond_alpha} is (up to equality) optimal.
  Furthermore,
  we note that requiring \eqref{eq:pw_ks:cond_M} for any $\pu \gt \frac n2(1-(m-q))$
  already implies global existence (cf.\ \cite[Theorem~2.2]{FreitagBlowupProfilesRefined2018}),
  while, to the best of our knowledge, even a solution blowing up in finite time might fulfill \eqref{eq:pw_ks:cond_M} for $\pu = \frac n2(1-(m-q))$.

  To sum up,
  \begin{align*}
    \text{optimal $L^p$ bounds imply essentially optimal pointwise upper estimates.}
  \end{align*}
\end{remark}

\paragraph{Notation}
Henceforth, we fix $n \ge 2$, $R \gt 0$ and $\Omega \defs B_R(0)$.
Moreover, with the usual slight abuse of notation, we switch to radial coordinates whenever convenient
and thus write for instance $v(|x|)$ for $v(x)$.

\section{Pointwise estimates for \tops{$\nabla v$}{grad v}. The elliptic case}\label{sec:elliptic}
We first deal with the much simpler elliptic case; that is, we set $\tau \defs 0$ in this section.
As a starting point, we obtain an $L^\qu$ bound for $\Delta v$ by a straightforward testing procedure.
For the parabolic case, which we will deal with in Section~\ref{sec:parabolic},
one cannot expect a similar result to hold if one only wants to assume that the initial datum satisfies \eqref{eq:cond_v0_p}
and not, say, $v_0 \in \sob22$ with $\partial_\nu v_0 = 0$ in the sense of traces and $\|v_0\|_{\sob22} \le M$.
\begin{lemma}\label{lm:ell_v_sob2p}
  Let $M \gt 0$ and $\qu \in [1, \infty)$.
  If $g$ is as in \eqref{eq:cond_g_e}
  and $v \in \con2$ is a classical solution of \eqref{prob:v_e},
  then
  \begin{align*}
    \|\Delta v\|_{\leb \qu} \le 2 M.
  \end{align*}
\end{lemma}
\begin{proof}
  Testing \eqref{prob:v_e} with $v^{\qu-1}$ and making use of Young's inequality gives
  \begin{align*}
        \intom v^\qu
    =   \intom v^{\qu-1} \Delta v
        + \intom v^{\qu-1} g
    \le -(\qu - 1) \intom v^{\qu-2} |\nabla v|^2
        + \frac{\qu-1}{\qu} \intom v^\qu + \frac1\qu \intom g^\qu
  \end{align*}
  and hence
  \begin{align*}
        \intom v^\qu
    \le \intom g^\qu
    \le M^\qu.
  \end{align*}

  For $\qu = 1$, this already implies
  \begin{align*}
    \intom |\Delta v| \le \intom (|v| + |g|) \le 2M,
  \end{align*}
  while for $\qu \gt 1$,
  we further test \eqref{prob:v_e} with $-\Delta v |\Delta v|^{\qu-2}$ and use Young's inequality to obtain
  \begin{align*}
        \intom |\Delta v|^\qu
    \le \intom (|v| + |g|) |\Delta v|^{\qu-1}
    \le \frac{\qu-1}{\qu} \intom |\Delta v|^\qu
        + \frac{2^{\qu-1}}{\qu} \intom |v|^\qu
        + \frac{2^{\qu-1}}{\qu} \intom |g|^\qu,
  \end{align*}
  which also implies
  \begin{equation*}
        \intom |\Delta v|^\qu
    \le 2^{\qu-1} \intom |v|^\qu
        + 2^{\qu-1} \intom |g|^\qu
    \le 2^\qu M^\qu,
  \end{equation*}
  as desired.
\end{proof}

Making crucial use of the radial symmetry,
we now show that the bound obtained in Lemma~\ref{lm:ell_v_sob2p} implies the desired estimate \eqref{eq:main_e:statement}.
\begin{lemma}\label{lm:ell_vr_est}
  Let $M \gt 0$, $\qu \in [1, n)$ and $\beta \ge \frac{n-\qu}{\qu}$.
  There is $C \gt 0$ such that if $g$ satisfies \eqref{eq:cond_g_e} and $v \in \con2$ is as a classical solution of \eqref{prob:v_e},
  then \eqref{eq:main_e:statement} holds.
\end{lemma}
\begin{proof}
  By the fundamental theorem of calculus, Hölder's inequality and Lemma~\ref{lm:ell_v_sob2p},
  we may calculate
  \begin{align}\label{eq:ell_vr_est:vr}
          r^{n-1} |v_r(r)|
    &=    \left| \int_0^r \rho^\frac{n-1}{\qu} \rho^{1-n} (\rho^{n-1} v_r)_r \cdot \rho^{-(n-1)\frac{1-\qu}{\qu}} \drho \right| \notag \\
    &\le  \frac{\|\Delta v\|_{\leb\qu}}{\sqrt[\qu]{\omega_{n-1}}}
          \left( \int_0^r \rho^{n-1} \drho \right)^{\frac{\qu-1}{\qu}}
     \le  \frac{2M n^{-\frac{\qu-1}{\qu}}}{\sqrt[\qu]{\omega_{n-1}}} 
          \cdot r^{n - \frac n\qu}
    \qquad \text{for all $r \in (0, R)$.}
  \end{align}
  In view of $r^{n - \frac n\qu - (n - 1)} = r^{-\frac{n-\qu}{\qu}} \le R^{\beta-\frac{n-\qu}{\qu}} r^{-\beta}$ for $r \in (0, R)$,
  dividing by $r^{n-1}$ on both the left and the right hand side in \eqref{eq:ell_vr_est:vr}
  implies \eqref{eq:main_e:statement} for an appropriately chosen $C \gt 0$.
\end{proof}

\section{Intermission: semigroup estimates}\label{sec:semigroup}
\mainspace{G}%
The proof of a parabolic counterpart to the preceding section will in multiple places rely on certain semigroup estimates,
which we collect here for convenience.
As we will apply them in both $\Omega$ and $(0, R)$, we consider arbitrary smooth bounded domains $G \subset \R^N$, $N \in \N$, in this section.

\begin{lemma}\label{lm:frac_powers}
  Let $G \subset \R^N$, $N \in \N$, be a smooth bounded domain,
  and $p \in (1, \infty)$.
  Set
  \begin{align*}
    \sobn 2p
    \defs \left\{\,\varphi \in \sob2p : \partial_\nu \varphi = 0 \text { on $\partial\Omega$ in the sense of traces}\,\right\}
  \end{align*}
  and define the operator $A$ on $\leb p$ by
  \begin{align*}
            A \varphi
    &\defs  A_p \varphi
     \defs  -\Delta \varphi + \varphi
    \qquad \text{for $\varphi \in \mc D(A) \defs \sobn2p$}.
  \end{align*}
  Define moreover the fractional powers $A^\mu$, $\mu \in (0, 1)$,
  of the operator above as in \cite[Section~1.15]{TriebelInterpolationTheoryFunction1978}.
  Then there are $C_1, C_2 \gt 0$ such that
  \begin{alignat*}{2}
    \|\varphi\|_{\sob{2\mu}{p}} &\le C_1 \|A^\mu \varphi\|_{\leb p}
      &&\qquad \text{for all $\varphi \in \mc D(A^\mu)$ and all $\mu \in (0, 1)$}
  \intertext{and}
    \|A^\mu \varphi\|_{\leb p} &\le C_2 \|\varphi\|_{\sob{2\mu}{p}}
      &&\qquad \text{for all $\varphi \in \sob{2\mu}{p}$ and all $\mu \in \left(0, \frac{1 + \frac1p}{2}\right)$}.
  \end{alignat*}
\end{lemma}
\begin{proof}
  Let $\mu \in (0, 1)$.
  From \cite[Theorems~1.15.3 and 4.3.3]{TriebelInterpolationTheoryFunction1978},
  we infer $\mc D(A^\mu) = [\leb p, \sobn2p]_{\mu} \subset H_p^{2\mu}(G)$ with equality if $2\mu \lt 1 + \frac1p$.
  (Herein, $[\cdot, \cdot]_\mu$ and $H_p^{2\mu}(G)$
  are as in \cite[Convention~1.9.2]{TriebelInterpolationTheoryFunction1978} and \cite[Definition~4.2.1]{TriebelInterpolationTheoryFunction1978},
  respectively.)
  Since $\Omega$ is smooth,
  \cite[Theorem~4.6.1~(d)]{TriebelInterpolationTheoryFunction1978} moreover asserts that $H_p^{2\mu}(G)$ coincides with $\sob{\mu}{p}$.
  Thus, we obtain the desired estimates by noting that $A^\mu$ is an isomorphism between $D(A^\mu)$ and $\leb p$
  (cf.\ \cite[Theorem~1.15.2~(e)]{TriebelInterpolationTheoryFunction1978}). 
\end{proof}
  
\begin{lemma}\label{lm:semigroup_est}
  Let $G \subset \R^N$, $N \in \N$, be a smooth bounded domain.
  \begin{enumerate}
    \item[(i)]
      Suppose $\sigma \in \{0, 1\}$, $\mu \in \R$,
      $q \in (1, \infty)$, $p \in [q, \infty]$ and
      \begin{align*}
        s
        \begin{cases}
          \ge \frac Nq - \frac Np, & p \lt \infty, \\
          \gt \frac{N}{q},         & p = \infty
        \end{cases}
      \end{align*}
      are such that $\mu + \frac{\sigma + s}{2} \ge 0$.
      For any $\lambda \in [0, \mu + \frac{\sigma + s}{2}] \cap [0, \frac12 + \frac1{2q})$ and $\delta \in (0, 1)$,
      we can then find $C \gt 0$
      \begin{align*}
              \|\nabla^\sigma A^\mu \ure^{-t A} \varphi\|_{\leb p}
        &\le  C t^{\lambda - \mu - \frac{\sigma + s}{2}} \ure^{-\delta t} \|\varphi\|_{\sob{2\lambda}q}
        \qquad \text{for all $t \gt 0$ and $\varphi \in \sob{2\lambda}{q}$},
      \end{align*}
      where $A = A_q$ is as in Lemma~\ref{lm:frac_powers}.
      (Here and below, $\nabla^0 = \id$ and $\nabla^1 = \nabla$.)

    \item[(ii)]
      In particular,
      for any $\sigma \in \{0, 1\}$, $\mu \in \R$ with $\mu + \frac{\sigma}{2} \ge 0$,
      $\lambda \in [0, \mu + \frac{\sigma}{2}] \cap [0, \frac12)$, $\delta \in (0, 1)$ and $\eps \in (0, 2N)$,
      there is $C' \gt 0$ such that
      \begin{align*}
              \|\nabla^\sigma A^\mu \ure^{-t A} \varphi\|_{\leb \infty}
        &\le  C' t^{\lambda - \mu - \frac{\sigma}{2} - \eps} \ure^{-\delta t} \|\varphi\|_{\con{2\lambda}}
        \qquad \text{for all $t \gt 0$ and $\varphi \in \con{2\lambda}$},
      \end{align*}
      where $A = A_q$ for a certain $q \in (1, \infty)$ is again as in Lemma~\ref{lm:frac_powers}.
  \end{enumerate}
\end{lemma}
\begin{proof}
  Let us first prove part~(i) for $s \lt 1$.
  To that end, we begin by fixing some constants:
  By \cite[Theorem~4.6.1~(c) and (e)]{TriebelInterpolationTheoryFunction1978}, there is $\newlc1 \gt 0$ such that
  \begin{alignat*}{2}
          \|\psi\|_{\leb p}
    &\le  \lc1 \|\psi\|_{\sob sq}
    && \qquad \text{for all $\psi \in \sob sq$}.
  \intertext{Moreover, noting that $\sigma + s \lt 2$, $2\lambda \lt 1 + \frac1q$ and $q \in (1, \infty)$,
  Lemma~\ref{lm:frac_powers} asserts that we can find $\newlc2, \newlc3 \gt 0$ with}
          \|\psi\|_{\sob{\sigma + s}q}
    &\le  \lc2 \|A^\frac{\sigma + s}{2}\psi\|_{\leb q}
    && \qquad \text{for all $\psi \in \mc D(A^\frac{\sigma+s}{2})$}
  \intertext{as well as}
          \|A^\lambda \psi\|_{\leb q}
    &\le  \lc3 \|\psi\|_{\sob{2\lambda}q}
    && \qquad \text{for all $\psi \in \sob{2\lambda}{q}$}
  \intertext{and \cite[Theorem~1.4.3]{HenryGeometricTheorySemilinear1981} provides us with $\newlc4 \gt 0$ such that}
          \|A^\gamma \ure^{t A} \psi\|_{\leb q}
    &\le  \lc4 t^{-\gamma} \ure^{-\delta t} \|\psi\|_{\leb q}
    && \qquad \text{for all $\psi \in \leb q$},
  \end{alignat*}
  where $\gamma \defs -\lambda + \mu + \frac{\sigma + s}{2} \ge 0$ by the assumption on $\lambda$.

  Moreover noting that $A^\mu \ure^{-t A} \varphi = \ure^{-\frac t2 A} A^\mu \ure^{- \frac t2 A} \varphi \in \mc D(A^\frac{\sigma+s}{2}) \cap \sob sq$
  for all $\varphi \in \leb p$,
  we may therefore estimate
  \begin{align*}
          \|\nabla^\sigma A^\mu \ure^{-t A} \varphi\|_{\leb p}
    &\le  \lc1 \|\nabla^\sigma A^\mu \ure^{-t A} \varphi\|_{\sob sq} \\
    &\le  \lc1 \|A^\mu \ure^{-t A} \varphi\|_{\sob{\sigma + s}q} \\
    &\le  \lc1 \lc2 \|A^{\frac{\sigma + s}{2} + \mu} \ure^{-t A} \varphi\|_{\leb q} \\
    &=    \lc1 \lc2 \|A^{-\lambda + \mu + \frac{\sigma + s}{2}} \ure^{-t A} A^\lambda \varphi\|_{\leb q} \\
    &\le  \lc1 \lc2 \lc4 t^{-\gamma} \|A^\lambda \varphi\|_{\leb q} \\
    &\le  \lc1 \lc2 \lc3 \lc4 t^{-\gamma} \|\varphi\|_{\sob{2\lambda} q}
    \qquad \text{for all $t \gt 0$ and $\varphi \in \sob{2\lambda}{q}$},
  \end{align*}
  which proves part~(i) if $s \lt 1$.
  If $s \in [1, \infty)$ and $p \lt \infty$,
  we fix $k \in \N$ and $p = p_0 \ge p_1 \ge \dots \ge p_k = q$ such that $s_j \defs \frac{N}{p_j} - \frac{N}{p_{j-1}} \lt 1$.
  Furthermore, we set
  \begin{align*}
    \mu_j \defs
    \begin{cases}
      - \frac{s_j}{2}, & j \lt k, \\ 
      \mu + \sum_{i=1}^{k-1} \frac{s_i}{2}, & j = k
    \end{cases}
    \qquad \text{for $j \in \{1, \dots, k\}$}
  \end{align*}
  and choose $\lambda$ to be $\frac{\sigma}{2}$ or $0$ (depending on whether the operator $\nabla^\sigma$ is involved) in first $k-1$ steps below.
  By the case already proven, we obtain then $\newlc{part1} \gt 0$ such that
  \begin{align*}
    &\pe  \left\| \nabla^\sigma A^\mu \ure^{-t A} \varphi \right\|_{\leb p} \\
    &=    \left\| \nabla^\sigma \prod_{j=1}^k \left( A^{\mu_j} \ure^{-\frac{t}{k} A} \right) \varphi \right\|_{\leb p} \\
    &\le  \lc{part1} \ure^{-\frac{\delta}{k} t}
            \left\| \prod_{j=2}^k \left( A^{\mu_j} \ure^{-\frac{t}{k} A} \right) \varphi \right\|_{\sob{\sigma}{p_1}} \\
    &\le  \lc{part1} \ure^{-\frac{\delta}{k} t} \left(
            \left\| \nabla^\sigma \prod_{j=2}^k \left( A^{\mu_j} \ure^{-\frac{t}{k} A} \right) \varphi \right\|_{\leb{p_1}}
            + \left\| \prod_{j=2}^k \left( A^{\mu_j} \ure^{-\frac{t}{k} A} \right) \varphi \right\|_{\leb{p_1}}
          \right) \\
    &\le  \lc{part1}^{k-1} \ure^{-\frac{(k-1)\delta}{k} t} \left(
            \left\| \nabla^\sigma A^{\mu_k} \ure^{-\frac{t}{k} A} \varphi \right\|_{\leb{p_{k-1}}}
            + (k-1) \|A^{-\frac{\sigma}{2}}_{p_{k-1}}\| \left\| A^{\mu_k + \frac{\sigma}{2}} \ure^{-\frac{t}{k} A} \varphi \right\|_{\leb{p_{k-1}}}
          \right) \\
    &\le  \lc{part1}^k (1 + (k-1) \|A^{-\frac{\sigma}{2}}_{p_{k-1}}\|)
            t^{\lambda - \mu_k - \frac{\sigma + s_k}{2}} \ure^{-\delta t} \|\varphi\|_{\sob{\sigma}{p_k}} \\
    &=    \lc{part1}^k (1 + (k-1) \|A^{-\frac{\sigma}{2}}_{p_{k-1}}\|)
            t^{\lambda - \mu - \frac{\sigma + s}{2}} \ure^{-\delta t} \|\varphi\|_{\sob{\sigma}{q}}
    \qquad \text{for all $t \gt 0$ and $\varphi \in \sob{2\lambda}{q}$},
  \end{align*}
  where in the last two steps we have made use of $\mu + \frac{\sigma+s}{2} = \mu_k + \frac{\sigma + s_k}{2}$.
  Finally, for $s \in [1, \infty)$ and $p = \infty$, the desired estimate follows from a similar iterative argument.

  Ad~(ii):
    Due to $\eps \in (0, 2N)$, we have $q \defs \frac{2N}{\eps} \in (1, \infty)$ and hence $s \defs \frac{2N}{q} = \eps$.
    We set moreover $p \defs \infty$ and $\tilde \lambda \defs \lambda - \frac{\eps}{2}$.
    Then the statement follows from part~(i) (with $\lambda$ replaced by $\tilde \lambda$)
    and the embedding $\sob{2\lambda}{q} \embed \con{2\lambda+\eps}$,
    which in turn directly follows from the fact that $\|\cdot\|_{\sob{2\lambda}{q}}$
    is equivalent to the norm given in~\cite[4.4.1~(8)]{TriebelInterpolationTheoryFunction1978}.
\end{proof}

While Lemma~\ref{lm:semigroup_est} is quite general,
its main shortcoming is the lack of $L^\infty$-$L^\infty$ estimates.
These are provided by the following lemma, at least for the special case $\mu = \lambda = 0$.
\begin{lemma}\label{lm:linfty_linfty}
  Letting $G \subset \R^N$, $N \in \N$, be a smooth bounded domain
  and defining the operator $A$ as in Lemma~\ref{lm:frac_powers},
  we can find $C \gt 0$ such that
  \begin{align*}
          \|\nabla^\sigma \ure^{-tA} \varphi\|_{\leb \infty}
    \le   C \ure^{-t} \|\nabla^\sigma \varphi\|_{\leb \infty}
    \qquad \text{for all $t \ge 0$, $\varphi \in \sob\sigma\infty$ and $\sigma \in \{0, 1\}$}.
  \end{align*}
\end{lemma}
\begin{proof}
  This immediately follows from the maximum principle and \cite[formula~(2.39)]{MoraSemilinearParabolicProblems1983}.
\end{proof}

\section{Pointwise estimates for \tops{$\nabla v$}{grad v}. The parabolic case}\label{sec:parabolic}
\mainspace{\Omega}%
In this section, we deal with the remaining case $\tau \gt 0$
and first argue that we may without loss of generality assume $\tau = 1$.
If $v \in C^0(\Ombar \times [0, T)) \cap C^{2, 1}(\Ombar \times (0, T))$ is a classical solution of \eqref{prob:v_p}
for some $\tau \gt 0$, $T \in (0, \infty]$, $v_0 \in \con 0$ and $g \in C^0(\Ombar \times [0, T))$,
then the function $\tilde v$ defined by $\tilde v(x, t) \defs v(x, \frac{t}{\tau})$ for $(x, t) \in \Ombar \times [0, T \tau)$
solves
\begin{align*}
  \begin{cases}
    \tilde v_t = \Delta \tilde v - \tilde v + \tilde g  & \text{in $\Omega \times (0, T \tau)$}, \\
    \partial_\nu \tilde v = 0                           & \text{on $\partial \Omega \times (0, T \tau)$}, \\
    \tilde v(\cdot, 0) = v_0                            & \text{in $\Omega$}
  \end{cases}
\end{align*}
classically,
where $\tilde g(x, t) \defs g(x, \frac{t}{\tau})$ for $(x, t) \in \Ombar \times [0, T \tau)$.
Since Theorem~\ref{th:main_p} requires $C$ to be independent of $T$
and $\sup_{t \in (0, \tau T)} \|\tilde g(\cdot, t)\|_{\leb \qu} = \sup_{t \in (0, T)} \|g(\cdot, t)\|_{\leb \qu}$ for all $\qu \ge 1$,
we may thus henceforth indeed fix $\tau = 1$ and prove Theorem~\ref{th:main_p} only for this special case.

Moreover, given $M \gt 0$, let us abbreviate 
\begin{align}\label{eq:cond}
  \begin{cases}
    v_0 \text{ and } g \text{ comply with \eqref{eq:cond_v0_p} and \eqref{eq:cond_g_p}}, \\
    v \in C^0(\Ombar \times [0, T)) \cap C^{2, 1}(\Ombar \times (0, T)) \text{ is a nonnegative classical solution of \eqref{prob:v_p}}.
  \end{cases}
\end{align}

Before proving Theorem~\ref{th:main_p} in Lemma~\ref{lm:nabla_v_pw} below, we first collect several estimates,
starting with an $\sob1p$ bound for certain $p \gt 1$.
\begin{lemma}\label{lm:v_sob1p}
  Let $M \gt 0$, $\qu \in [1, n]$, $p_0 \gt 1$ and $p \in (1, \frac{n \qu}{n - \qu}) \cap (1, p_0]$.
  There is $C \gt 0$ such that if \eqref{eq:cond} holds, then
  \begin{align}\label{eq:v_sob1p:statement}
    \|\nabla v(\cdot, t)\|_{\leb p} \le C \qquad \text{for all $t \in (0, T)$}.
  \end{align}
\end{lemma}
\begin{proof}
  Letting $A$ be as in Lemma~\ref{lm:frac_powers},
  we apply Lemma~\ref{lm:semigroup_est} (with
  $\sigma \defs 1$, $\mu \defs \frac12$, $q \defs p$, $s \defs 0$, $\lambda \defs \frac12$ and
  $\sigma \defs 1$, $\mu \defs 0$, $q \defs \qu$, $s \defs \frac nq - \frac np$, $\lambda \defs 0$)
  to obtain $\newlc{1}, \newlc{2} \gt 0$ and $\delta \gt 0$ such that
  \begin{alignat*}{2}
          \|\nabla \ure^{-t A} \varphi\|_{\leb p}
    &\le  \lc{1} \ure^{-\delta t} \|\nabla \varphi\|_{\leb p}
    && \qquad \text{for all $t \gt 0$ and $\varphi \in \sob1p$},
  \intertext{and}
          \|\nabla \ure^{-t A} \varphi\|_{\leb p}
    &\le  \lc{2} t^{-\frac12 - \frac n2 (\frac1\qu - \frac1p)} \ure^{-\delta t} \|\varphi\|_{\leb \qu}
    && \qquad \text{for all $t \gt 0$ and $\varphi \in \leb \qu$}.
  \end{alignat*}
  Hence, assuming \eqref{eq:cond},
  we make use of the variation-of-constants formula, \eqref{eq:cond_v0_p} and \eqref{eq:cond_g_p} to see that
  \begin{align*}
          \|\nabla v(\cdot, t)\|_{\leb p}
    &\le  \left\| \nabla \ure^{-tA} v_0 \right\|_{\leb p}
          + \int_0^t \left\| \ure^{-(t-s) A} g(\cdot, s) \right\|_{\leb p} \ds \\
    &\le  \lc{1} \ure^{-\delta t} \|\nabla v_0\|_{\leb p}
          + \lc{2} \|g\|_{L^\infty((0, T); \leb\qu)} \int_0^t (t-s)^{-\frac12 - \frac n2 (\frac1\qu - \frac1p)} \ure^{-\delta(t-s)} \ds \\
    &\le  M \lc{1} |\Omega|^\frac{p_0}{p_0-p}
          + M \lc{2} \int_0^\infty s^{-\frac12 - \frac n2 (\frac1\qu - \frac1p)} \ure^{-\delta s} \ds
    \qquad \text{for all $t \in (0, T)$}.
  \end{align*}
  The last integral therein is finite because the assumption $p \lt \frac{n \qu}{n - \qu}$ warrants
  \begin{equation*}
        - \frac12 - \frac n2 \left( \frac1\qu - \frac1p \right)
    \gt - \frac12 - \frac n2 \left( \frac{n}{n\qu} - \frac{n-\qu}{n\qu} \right)
    =   -1.
    \qedhere
  \end{equation*}
\end{proof}

If $\qu \in [1, \frac n2]$, then the gradient bound obtained in Lemma~\ref{lm:v_sob1p} implies certain pointwise upper bounds for $v$.
For the special case $\qu = 1$, this has already been proven (similarly as below) in \cite[Lemma~3.2]{WinklerFinitetimeBlowupHigherdimensional2013}.
\begin{lemma}\label{lm:v_upper_est}
  Given $M \gt 0$, $\qu \in [1, \frac n2]$, $p_0 \gt 1$ and $\kappa \in (-\infty, -\frac{n - 2\qu}{\qu}) \cap (-\infty, -\frac{n-p_0}{p_0}]$,
  there is $C \gt 0$ with the following property:
  If $T \in (0, \infty]$ and \eqref{eq:cond} holds, then
  \begin{align*}
    v(x, t) \le C |x|^{\kappa} \qquad \text{for all $x \in \Ombar$ and $t \in (0, T)$}.
  \end{align*}
\end{lemma}
\begin{proof}
  For fixed $\kappa \le -\frac{n-p_0}{p_0}$ with
  \begin{align*}
        \kappa
    \lt - \frac{n - 2\qu}{\qu}
    =   - \frac{(n - \qu) - \qu}{\qu}
    =   -\frac{n - \frac{n \qu}{n - \qu}}{\frac{n \qu}{n - \qu}},
  \end{align*}
  we may choose $p \in (1, \frac{n \qu}{n - \qu}) \cap (1, p_0]$ such that $\kappa \le -\frac{n-p}{p}$.
  Then Lemma~\ref{lm:v_sob1p} warrants that there is $\newlc{v_sob1p} \gt 0$ such that \eqref{eq:v_sob1p:statement} (with $C$ replaced by $\lc{v_sob1p}$)
  is fulfilled whenever \eqref{eq:cond} holds.
  Moreover, we let
  \begin{align*}
    \newlc1 \defs M\max\left\{|\Omega|^\frac{p_0-1}{p_0}, |\Omega|^\frac{\qu-1}{\qu}\right\}
    \quad \text{as well as} \quad
    \newlc2 \defs \frac{\lc1}{\left|B_R(0) \setminus B_\frac R2(0)\right|}
  \end{align*}
  and now assume \eqref{eq:cond}. 
  Since
  \begin{align*}
    \|v_0\|_{\leb 1} \le |\Omega|^\frac{p_0-1}{p_0} \|v_0\|_{\sob1{p_0}} \le \lc1
    \quad \text{and} \quad
    \|g\|_{L^\infty((0, T); \leb1)} \le |\Omega|^\frac{\qu-1}{\qu} \|g\|_{L^\infty((0, T); \leb\qu)} \le \lc1,
  \end{align*}
  by \eqref{eq:cond_v0_p}, \eqref{eq:cond_g_p} and the definition of $\lc1$,
  the comparison principle asserts $\intom v(\cdot, t) \le \lc1$ for all $t \in [0, T)$.

  Thus, assuming that there is $t \in [0, T)$ such that $v(r, t) \gt \lc2$ for all $r \in (\frac R2, R)$ would lead to the contradiction
  \begin{align*}
        \lc1
    \ge \intom v(\cdot, t)
    \ge \int_{B_R(0) \setminus B_\frac R2(0)} v(\cdot, t)
    \gt \int_{B_R(0) \setminus B_\frac R2(0)} \lc2
    =   \lc1,
  \end{align*}
  and therefore, for all $t \in [0, T)$, we may choose $r_0(t) \in (\frac R2, R)$ with $v(r_0(t), t) \le \lc2$.
  We then calculate
  \begin{align*}
          v(r, t) - v(r_0(t), t)
    &=    \int_{r_0(t)}^r \rho^\frac{n-1}{p} v_r(\rho, t) \cdot \rho^{-\frac{n-1}{p}} \drho \\
    &\le  \frac{\|\nabla v(\cdot, t)\|_{\leb p}}{\sqrt[p]{\omega_{n-1}}}
          \left| \int_{r_0(t)}^r \rho^{-\frac{n-1}{p-1}} \right|^\frac{p-1}{p} \\
    &\le  \frac{\lc{v_sob1p}}{\sqrt[p]{\omega_{n-1}}}
          \left| \int_{r_0(t)}^r \rho^{-\frac{n-1}{p-1}} \right|^\frac{p-1}{p}
    \qquad \text{for all $r \in (0, R)$ and $t \in (0, T)$}.
  \end{align*}
  
  As $p \in (1, n)$ because of $\qu \le \frac n2$ and $\frac{n \qu}{n - \qu} \le n$
  and since $r_0(t) \gt \frac{R}{2} \ge \frac{r}{2}$ for all $r \in (0, R)$ and $t \in (0, T)$,
  we have therein
  \begin{align*}
          \left| \int_{r_0(t)}^r \rho^{-\frac{n-1}{p-1}} \right|^\frac{p-1}{p}
    &\le  \left( \int_{\min\{r, r_0(t)\}}^\infty \rho^{-\frac{n-p}{p-1}-1} \right)^\frac{p-1}{p} \\
    &=    \left( \frac{p-1}{n-p} \right)^\frac{p-1}{p} \min\{r, r_0(t)\}^{-\frac{n-p}{p}} \\
    &\le  2^\frac{n-p}{p} \left( \frac{p-1}{n-p} \right)^\frac{p-1}{p} r^{-\frac{n-p}{p}}
    \qquad \text{for all $r \in (0, R)$ and $t \in (0, T)$}.
  \end{align*}
  Moreover noting that $v(r_0(t), t) \le \lc2 \le \lc2 R^\frac{n-p}{p} r^{-\frac{n-p}{p}}$ for all $r \in (0, R)$ and $t \in (0, T)$,
  we obtain the statement.
\end{proof}

Since $\qu \gt \frac n2$ implies $\frac{2\qu - n}{\qu} \gt 0$,
one cannot expect that Lemma~\ref{lm:v_upper_est} holds for any $\qu \gt \frac n2$.
However, we have the following analogon of said lemma.

\begin{lemma}\label{lm:v_hoelder}
  For $M \gt 0$, $\qu \in (\frac n2, n]$, $p_0 \gt 1$ and $\kappa \in (0,  \frac{2\qu - n}{\qu}) \cap (0, \frac{p_0-n}{p_0}]$,
  there is $C \gt 0$ such that if $T \in (0, \infty]$ and \eqref{eq:cond} holds, then
  \begin{align*}
    |v(x, t) - v(0, t)| \le C |x|^\kappa
    \qquad \text{for all $x \in \Ombar$ and $t \in [0, T)$}.
  \end{align*}
\end{lemma}
\begin{proof}
  Let $\kappa \in (0, \frac{2\qu - n}{\qu})$.
  The assumption $\qu \in (\frac n2, n]$ implies $\kappa \in (0, 1)$,
  hence $p \defs \frac{n}{1-\kappa} \in (1, \frac{n\qu}{n-\qu}) \cap (1, p_0]$.
  Thus, the statement follows from Lemma~\ref{lm:v_sob1p} and Morrey's inequality,
  which because of $\kappa = 1 - \frac np$ asserts that $\sob1p$ embeds into $\con\kappa$.
\end{proof}

Lemma~\ref{lm:v_hoelder} now allows us to show that a function resembling $|x|^\beta v$
solves a suitable initial boundary value problem.
In Lemma~\ref{lm:nabla_v_pw} below, we then apply semigroup arguments to obtain certain gradient bounds for this function
implying \eqref{eq:main_p:statement}.
\begin{lemma}\label{lm:z_eq}
  Let $M \gt 0$, $\qu \in [1, n]$, $\beta \gt \frac{n-\qu}{\qu}$,
  \begin{align}\label{eq:z_eq:zeta}
    \zeta \in C^\infty([0, R]) \text{ with  $\zeta(r) = r$ for all $r \in [0, \tfrac R2]$, $\zeta_r  \ge 0$ in $(0, R)$ and $\zeta_r(R) = 0$}
  \end{align}
  and
  \begin{align*}
    p_0 \gt
    \begin{cases}
      1, & \qu \in [1, \frac n2], \\
      \frac{n}{\min\{1, \beta\}}, & \qu \in (\frac n2, n].
    \end{cases}
  \end{align*}
  There exist $b_1, b_2, b_3 \in C^\infty((0, R))$ and $C \gt 0$ such that
  \begin{align}\label{eq:z_eq:b_est}
    |b_1(r)| \le C r^{\beta-2}, \qquad
    |b_2(r)| \le C r^{\beta-1} \qquad \text{and} \qquad 
    |b_3(r)| \le C r^{\beta}
    \qquad\text{for all $r \in (0, R)$},
  \end{align}
  and, moreover, the following holds:
  Let $T \in [0, \infty)$, $v_0, g, v$ as in \eqref{eq:cond} and
  \begin{align}\label{eq:z_eq:def_v_tilde}
    \tilde v(r, t) \defs
    \begin{cases}
      v(r, t),           & \qu \in [1, \frac n2], \\
      v(r, t) - v(0, t), & \qu \in (\frac n2, n]
    \end{cases}
    \qquad \text{for $r \in [0, R]$ and $t \in [0, T)$}.
  \end{align}
  Then the function $z \defs \zeta^\beta \tilde v$
  belongs to
  $C^0([0, R] \times [0, T)) \cap C^{1, 1}([0, R] \times (0, T)) \cap C^{2, 1}((0, R) \times (0, T))$
  and solves
  \begin{align}\label{eq:z_eq:prob}
    \begin{cases}
      z_t = z_{rr} - z + b_1 \tilde v + b_2 v_r + b_3 g - [\sign(\qu - \frac n2)]_+ \zeta^\beta v_t(0, t) , & \text{in $(0, R) \times (0, T)$}, \\
      z_r = 0,                                                                         & \text{in $\{0, R\} \times (0, T)$}, \\
      z(\cdot, 0) = \zeta^\beta \tilde v(\cdot, 0)                                     & \text{in $(0, R)$}
    \end{cases}
  \end{align}
  classically.
  (Here and below, $[\sign \xi]_+ = 1$ for $\xi \gt 0$ and $[\sign \xi]_+ = 0$ for $\xi \le 0$.)
\end{lemma}
\begin{proof}
  Since the assumptions on $\zeta$ warrant $\|\zeta\|_{\con[{[0,R]}]2} \lt \infty$ and $\sup_{r \in (0, R)} \frac{\zeta(r)}{r} \lt \infty$,
  there is $C \gt 0$ such that the functions
  \begin{align*}
    b_1 &\defs -\beta (\beta-1) \zeta^{\beta-2} \zeta_r^2 - \beta \zeta^{\beta-1} \zeta_{rr} \\
    b_2 &\defs -2 \beta \zeta^{\beta-1} \zeta_r + \frac{n-1}{r} \zeta^\beta \qquad \text{and} \\
    b_3 &\defs \zeta^\beta
  \end{align*}
  comply with \eqref{eq:z_eq:b_est}.
  As direct calculations give
  \begin{align*}
        z_r
    &=  \beta \zeta^{\beta-1} \zeta_r \tilde v
        + \zeta^\beta v_r, \\
        z_{rr}
    &=  [\beta (\beta-1) \zeta^{\beta-2} \zeta_r^2 + \beta \zeta^{\beta-1} \zeta_{rr}] \tilde v
        + 2 \beta \zeta^{\beta-1} \zeta_r v_r
        + \zeta^{\beta} v_{rr} \qquad \text{and} \\
        v_t
    &=  v_{rr} + \frac{n-1}{r} v_r - v + g
  \end{align*}
  in $(0, R) \times (0, T)$,
  we obtain moreover
  \begin{align*}
        \zeta^\beta v_t
    &=  \zeta^\beta v_{rr} + \frac{n-1}{r} \zeta^\beta v_r - \zeta^\beta v + \zeta^\beta g \\
    &=  z_{rr} - \left[\beta (\beta-1) \zeta^{\beta-2} \zeta_r^2 + \beta \zeta^{\beta-1} \zeta_{rr}\right] \tilde v
        + \left[-2 \beta \zeta^{\beta-1} \zeta_r + \frac{n-1}{r} \zeta^\beta\right] v_r
        - z
        + \zeta^\beta g
    \qquad \text{in $(0, R) \times (0, T)$}.
  \end{align*}
  Thus,
  \begin{align*}
        z_t(r, t)
    &=  \zeta^\beta(r) v_t(r, t) - \left[\sign\left(\qu - \frac n2\right)\right]_+ \zeta^\beta(r) v_t(0, t)
    \qquad \text{for all $(r, t) \in [0, R) \times [0, T)$},
  \end{align*}
  implying that the first equation in \eqref{eq:z_eq:prob} holds.

  Since the third equation in \eqref{eq:z_eq:prob} is a direct consequence of the definition of $z$
  and $\zeta_r(R) = 0$ and $v_r(R, \cdot) \equiv 0$ and $\zeta(R) \gt 0$ imply $z_r(R, \cdot) \equiv 0$,
  it only remains to be shown that $z_r(0, \cdot) \equiv 0$ in $(0, T)$.
  For $\qu \in [1, \frac n2]$ and hence $\beta \gt 1$, this holds because then $\lim_{r \sea 0} \zeta^{\beta-1}(r) = 0$.
  Thus, we suppose now that $\qu \in (\frac n2, n]$.
  As $\frac{2\qu - n}{\qu} \gt \max\{1 - \beta, 0\}$ and $\frac{p_0 - n}{p_0} \gt \max\{1 - \beta, 0\}$,
  we may choose $\kappa \in (\max\{1 - \beta, 0\}, \min\{\frac{2\qu - n}{\qu}, \frac{p_0 - n}{p_0}\})$
  and apply Lemma~\ref{lm:v_hoelder} to obtain $\newlc{v_time_hoelder} \gt 0$
  such that $|v(r, t) - v(0, t)| \le \lc{v_time_hoelder} r^{\kappa}$ for all $(r, t) \in (0, R) \times (0, T)$.
  Thus, $|\zeta^{\beta-1}(r) \tilde v(r, t)| \le \lc{v_time_hoelder} r^{\beta - 1 + \kappa} \ra 0$ as $\frac R2 \ge r \sea 0$.
\end{proof}

For $\qu \in (\frac n2, n]$, we need to handle the term $\zeta^\beta v_t(0, \cdot)$ in the first equation in \eqref{eq:z_eq:prob}
if we want to apply semigroup arguments to the problem \eqref{eq:z_eq:prob}.
To that end, we argue similar as in \cite[Lemma~3.4]{WangEtAlGlobalClassicalSolutions2018}
and derive sufficiently strong time regularity in
\begin{lemma}\label{lm:v_time_hoelder}
  Suppose $M \gt 0$, $\qu \in (\frac n2, n]$, $p_0 \gt n$ and $\theta \in (0, \min\{\frac{2\qu - n}{2\qu}, \frac{p_0 - n}{2p_0}\})$.
  Then there exists $C \gt 0$ such that for $T \in (0, \infty]$ and $v_0, g, v$ complying with \eqref{eq:cond}, we have
  \begin{align}\label{eq:v_time_hoelder:statement}
    |v(0, t_1) - v(0, t_2)| \le C |t_1 - t_2|^\theta
    \qquad \text{for all $t_1, t_2 \in [0, T)$}.
  \end{align}
\end{lemma}
\begin{proof}
  Since $0 \lt \theta \lt \frac12 - \frac n{2p_0}$,
  we can choose $p \in (1, p_0)$ and $\eps \gt 0$ such that $\theta = \frac12 - \frac{n}{2p} - \eps$.
  Letting $A$ be as in Lemma~\ref{lm:frac_powers},
  by Lemma~\ref{lm:frac_powers} and Lemma~\ref{lm:semigroup_est}~(i)
  (with $\sigma \defs 0$, $\mu \defs \frac12$, $q \defs p$, $p \defs \infty$, $s \defs \frac{n}{q} + \eps$, $\lambda \defs 0$),
  we find $\newlc{sg1}, \newlc{sg2} \gt 0$ such that
  \begin{alignat}{2}\label{eq:v_time_hoelder_c1}
          \|A^\frac12 \varphi\|_{\leb p}
    &\le  \lc{sg1} \|\varphi\|_{\sob1p}
    &&\qquad \text{for all $\varphi \in \sob1p$}
  \intertext{and}\label{eq:v_time_hoelder_c2}
          \left\| A^\frac12 \ure^{-t A} \varphi \right\|_{\leb \infty}
    &\le  \lc{sg2} t^{-\frac12 - \frac{n}{2p} - \eps} \|\varphi\|_{\leb p}
    &&\qquad \text{for all $t \gt 0$, $\varphi \in \leb p$}.
  \intertext{Moreover, since $1 - \theta \gt \frac{n}{2\qu}$,
  we may again employ Lemma~\ref{lm:semigroup_est}~(i)
  (with $\sigma \defs 0$, $\mu \defs \mu$, $q \defs \qu$, $p \defs \infty$, $s \defs 1 - \theta$, $\lambda \defs 0$)
  in order to obtain $\newlc{sg3} \gt 0$ with}
          \left\| A^\mu \ure^{-t A} \varphi \right\|_{\leb \infty} \label{eq:v_time_hoelder_c3}
    &\le  \lc{sg3} t^{-\mu - (1 - \theta)} \|\varphi\|_{\leb \qu}
    &&\qquad \text{for all $t \gt 0$, $\varphi \in \leb \qu$ and $\mu \in \{0, 1\}$}.
  \end{alignat}

  Henceforth fixing $0 \le t_1 \lt t_2 \lt T$ and assuming \eqref{eq:cond},
  we then obtain by the variation-of-constants formula
  \begin{align*}
    &\pe  \left\| v(\cdot, t_2) - v(\cdot, t_1) \right\|_{\leb\infty} \\
    &\le  \left\| \ure^{-t_2 A} v_0 - \ure^{-t_1 A} v_0 \right\|_{\leb\infty}
          + \left\|
            \int_0^{t_2} \ure^{-(t_2-s) A} g(\cdot, s) \ds
            - \int_0^{t_1} \ure^{-(t_1-s) A} g(\cdot, s) \ds
          \right\|_{\leb\infty} \\
    &\le  \left\| \ure^{-t_2 A} v_0 - \ure^{-t_1 A} v_0 \right\|_{\leb\infty} \\
    &\pe  + \int_{t_1}^{t_2} \left\| \ure^{-(t_2-s) A} g(\cdot, s) \right\|_{\leb \infty} \ds
          + \int_0^{t_1} \left\| \left[ \ure^{-(t_2-s) A} - \ure^{-(t_1-s) A} \right] g(\cdot, s) \right\|_{\leb\infty} \ds \\
    &\sfed I_1 + I_2 + I_3.
  \end{align*}
  
  Firstly,
  due to the fundamental theorem of calculus,
  since $A^\frac12 \ure^{-t A} = \ure^{-t A} A^\frac12$ on $\mc D(A)$ for all $t \ge 0$,
  and because of
  \eqref{eq:v_time_hoelder_c2},
  \eqref{eq:v_time_hoelder_c1},
  the definition of $\theta$
  and \eqref{eq:cond_v0_p},
  we have therein
  \begin{align*}
          I_1
    &=    \left\| \int_{t_1}^{t_2} A \ure^{-s A} v_0 \ds \right\|_{\leb \infty} \\
    &\le  \int_{t_1}^{t_2} \left\|A^{\frac12} \ure^{-s A} A^\frac12 v_0 \right\|_{\leb \infty} \ds \\
    &\le  \lc{sg2} \|A^\frac12 v_0 \|_{\leb p} \int_{t_1}^{t_2} s^{-\frac12 - \frac{n}{2p} - \eps} \ds \\
    &\le  \frac{\lc{sg1} \lc{sg2} \|v_0\|_{\sob1p}}\theta (t_2 - t_1)^\theta
     \le  \frac{M \lc{sg1} \lc{sg2} |\Omega|^\frac{p_0-p}{p_0}}\theta (t_2 - t_1)^\theta,
  \end{align*}
  secondly, \eqref{eq:v_time_hoelder_c3}, the fundamental theorem of calculus and \eqref{eq:cond_g_p} imply
  \begin{align*}
          I_2
    &=    \int_{t_1}^{t_2} \left\| \ure^{-(t_2-s) A} g(\cdot, s) \right\|_{\leb \infty} \ds
     \le  \lc{sg3} \int_{t_1}^{t_2} (t_2-s)^{\theta-1} \left\|g(\cdot, s) \right\|_{\leb \qu} \ds
     \le  \frac{M \lc{sg3}}{\theta} (t_2 - t_1)^\theta
  \end{align*}
  and thirdly, from \eqref{eq:v_time_hoelder_c3}, the fundamental theorem of calculus, \eqref{eq:cond_g_p} and the fact that $t_2 \gt t_1$,
  we infer
  \begin{align*}
          I_3
    &=    \int_0^{t_1} \int_{t_1}^{t_2} \left\| A \ure^{-(\sigma-s) A} g(\cdot, s) \right\|_{\leb\infty} \dsigma \ds \\
    &\le  \lc{sg3} \int_0^{t_1} \int_{t_1}^{t_2} (\sigma-s)^{\theta-2} \|g(\cdot, s)\|_{\leb \qu} \dsigma \ds \\
    &\le  - \frac{M \lc{sg3}}{1-\theta} \int_0^{t_1} \left[ (t_2-s)^{\theta-1} - (t_1-s)^{\theta-1} \right] \ds \\
    &=    \frac{M \lc{sg3}}{\theta(1-\theta)} \left[ (t_2 - t_1)^\theta - t_2^\theta + t_1^\theta \right]
     \le  \frac{M \lc{sg3}}{\theta(1-\theta)} (t_2 - t_1)^\theta.
  \end{align*}
  Together, this implies \eqref{eq:v_time_hoelder:statement}.
\end{proof}

We now combine the estimates gathered above to prove Theorem~\ref{th:main_p}.
\mainspace{(0, R)}
\renewcommand{\con}[2][{[0, R]}]{\ensuremath{C^{#2}(#1)}}
\begin{lemma}\label{lm:nabla_v_pw}
  Let $M \gt 0$, $\qu \in (1, n]$, $\beta \gt \frac{n-\qu}{\qu}$ and $p_0 \gt \max\{\frac{n}{\beta}, 1\}$.
  There exists $C \gt 0$ such that whenever $T \in (0, \infty]$ and $v_0, g, v$ satisfy \eqref{eq:cond},
  then \eqref{eq:main_p:statement} holds.
\end{lemma}
\begin{proof}
  For $\qu \in (1, \frac n2]$ and $\qu \in (\frac n2, n]$, we assume without loss of generality $\beta \in [1, n)$ and $\beta \in (0, 1)$, respectively.
  Moreover, the assumptions on the parameters allow us to choose $\bar p \in (\max\{\frac n\beta, 1\}, \min\{\frac{n \qu}{n - \qu}, p_0\})$
  and
  \begin{align}\label{eq:nabla_v_pw:def_kappa}
    \kappa \in \left(1- \beta, \min\left\{\frac{2\qu-n}{\qu}, \frac{p_0-n}{p_0}\right\} \right).
  \end{align}
  
  Noting that $\bar p \gt \max\{\frac{n}{\beta}, 1\}$ and hence
  \begin{align*}
        \frac{(\beta - 1) \bar p - (n - 1)}{\bar p - 1}
    \gt \frac{1 - \bar p}{\bar p - 1}
    =   -1
  \end{align*}
  hold,
  that $\kappa \gt 1 - \beta$ implies $\beta - 2 + \kappa \gt - 1$
  and that the main assumption, $\beta \gt \frac{n - \qu}{\qu}$, asserts
  \begin{align*}
        \frac{\beta \qu - (n-1)}{\qu - 1}
    \gt \frac{-(\qu-1)}{\qu - 1}
    =   -1,
  \end{align*}
  we can find $p \in (1, \min\{\bar p, \qu\})$ such that still
  \begin{align}\label{eq:nabla_v_pw:lambda}
          \lambda_1
    \defs (\beta - 2 + \kappa) p
    \gt   -1,
    \quad
          \lambda_2
    \defs \frac{[(\beta - 1) \bar p - (n - 1)] p}{\bar p - p}
    \gt   -1
    \quad \text{and} \quad
          \lambda_3
    \defs \frac{[\beta \qu - (n-1)]p}{\qu - p}
    \gt   -1.
  \end{align}

  Letting now $A$ be as in Lemma~\ref{lm:frac_powers} with $G \defs (0, R)$,
  Lemma~\ref{lm:linfty_linfty} and Lemma~\ref{lm:semigroup_est}~(i)
  allow us to fix $\newlc{sg1}, \newlc{sg2} \gt 0$ and $\delta_1 \gt 0$ such that
  \begin{alignat}{2}\label{eq:nabla_v_pw:c_1}
          \|\partial_r \ure^{-\tau A} \varphi\|_{\leb\infty}
    &\le  \lc{sg1} \ure^{-\tau} \|\varphi_r\|_{\leb\infty}
    &&\qquad \text{for all $\varphi \in \sob1\infty$ and all $\tau \gt 0$}
  \intertext{and}\label{eq:nabla_v_pw:c_2}
          \|\partial_r \ure^{-\tau A} \varphi\|_{\leb\infty}
    &\le  \lc{sg2} \tau^{\gamma_1} \ure^{-\delta_1 \tau} \|\varphi\|_{\leb p}
    &&\qquad \text{for all $\varphi \in \leb[(0, R)]p$ and all $\tau \gt 0$},
  \end{alignat}
  where $\gamma_1 \defs -\frac12 - \frac{p+1}{4p}$.
  (We note that $\frac{p+1}{2p} \gt \frac1p$ because of $p \gt 1$, so that Lemma~\ref{lm:semigroup_est} is indeed applicable.)
  Since $p \gt 1$, we have $\gamma_1 \gt -1$ and hence
  \begin{align}\label{eq:nabla_v_pw:second_int}
          \newlc{second_int} 
    \defs \sup_{t \in (0, \infty)} \int_0^t (t-s)^{\gamma_1} \ure^{-\delta_1 (t-s)} \ds 
    =     \int_0^\infty s^{\gamma_1} \ure^{-\delta_1 s} \ds 
    \lt   \infty.
  \end{align}
 
  Moreover, by Lemma~\ref{lm:v_sob1p}, Lemma~\ref{lm:v_upper_est} and Lemma~\ref{lm:v_hoelder},
  there are $\newlc{v_w1p}, \newlc{v_pw} \gt 0$ such that
  \begin{align}\label{eq:nabla_v_pw:v_w1p_pw}
    \|\nabla v(\cdot, t)\|_{\leb[\Omega]{\bar p}} \le \lc{v_w1p}
    \quad \text{and} \quad
    |\tilde v(x, t)| \le \lc{v_pw} |x|^{\kappa}
    \qquad \text{for all $x \in \Omega$ and $t \in [0, T)$},
  \end{align}
  whenever \eqref{eq:cond} is fulfilled and
  where $\tilde v$ is given by \eqref{eq:z_eq:def_v_tilde}.

  If $\qu \in (\frac n2, n]$,
  due to $\frac{2\qu-n}{2\qu} + \frac{\beta}{2} \gt \frac{2\qu-n}{2\qu} + \frac{n-\qu}{2\qu} = \frac12$,
  we may also choose $\eps \in (0, 2)$ and $\theta \in (0, \frac{2\qu - n}{2\qu})$ sufficiently small and large, respectively, such that
  \begin{align*}
    \gamma_2 \defs \theta + \frac{\beta}{2} - \frac32 - \eps \gt - 1.
  \end{align*}
  Since $\qu \in (\frac n2, n]$ implies $\beta \in (0, 1)$,
  an application of Lemma~\ref{lm:semigroup_est}~(ii) then yields $\newlc{sg3} \gt 0$ and $\delta_2 \gt 0$ such that for $\mu \in \{0, 1\}$,
  \begin{alignat}{2}\label{eq:nabla_v_pw:sg3}
          \|\partial_r A^\mu \ure^{-\tau A} \varphi\|_{\leb\infty}
    &\le  \lc{sg3} \tau^{\frac{\beta}{2} - \mu - \frac12 - \eps} \ure^{-\delta_2 \tau} \|\varphi\|_{\con\beta}
    &&\qquad \text{for all $\varphi \in \sob1\infty$ and all $\tau \gt 0$}.
  \end{alignat}
  Furthermore, again only in the case $\qu \in (\frac n2, n]$,
  Lemma~\ref{lm:v_time_hoelder} allows us to fix $\newlc{v_hoelder_time} \gt 0$ such that
  \begin{alignat}{2}\label{eq:nabla_v_pw:v_hoelder_time}
          |v(0, t_2) - v(0, t_1)|
    &\le  \lc{v_hoelder_time} |t_2 - t_1|^\theta
    &&\qquad \text{for all $t_1, t_2 \in (0, T)$}
  \end{alignat}
  and (provided $\qu \in (\frac n2, n]$) we set
  \begin{align}\label{eq:nabla_v_pw:last_int}
          \newlc{last_int}
    \defs \int_0^\infty s^{\gamma_2} \ure^{-\delta_2 s} \ds 
          + \sup_{t \in (0, \infty)} t^{\gamma_2 + 1} \ure^{-\delta_2 t}
    \lt   \infty.
  \end{align}

  As a last preparation, regardless of the sign of $\qu - \frac n2$,
  we fix an arbitrary $\zeta$ as in \eqref{eq:z_eq:zeta}.
  Hence there are $\newlc{zeta}, \newlc{zeta_r}, \newlc{zeta_hoelder} \gt 0$ with
  \begin{align}\label{eq:nabla_v_pw:zeta}
    \frac{r}{\lc{zeta}} \le \zeta(r) \le \lc{zeta} r,
    \quad
    |\zeta_r(r)| \le \lc{zeta_r}
    \quad \text{and} \quad
    \|\zeta^\beta\|_{\con\beta} \le \lc{zeta_hoelder}
    \qquad \text{for all $r \in (0, R)$}
  \end{align}
  and, by Lemma~\ref{lm:z_eq}, there is moreover $\newlc{b} \gt 0$ such that \eqref{eq:z_eq:b_est} holds (with $C$ replaced by $\lc{b}$),
  where $b_1, b_2, b_3$ are also given by Lemma~\ref{lm:z_eq}.

  We suppose now \eqref{eq:cond}.
  Noting that $\beta \gt \frac{n-\qu}{\qu}$,
  we may infer from Lemma~\ref{lm:z_eq} that $z \defs \zeta^\beta \tilde v$ is a classical solution of \eqref{eq:z_eq:prob}.
  By the variation-of-constants formula, we may therefore write
  \begin{align*}
          \|z_r(\cdot, t)\|_{\leb \infty}
    &\le  \|\partial_r \ure^{-t A} z(\cdot, 0)\|_{\leb \infty} \\
    &\pe  + \int_0^t \| \partial_r \ure^{-(t-s) A} [b_1 \tilde v(\cdot, s) + b_2 v_r(\cdot, s) + b_3 g(\cdot, s)] \|_{\leb \infty} \ds \\
    &\pe  + \left[\sign\left(\qu - \frac n2\right)\right]_+ \int_0^t \| \partial_r \ure^{-(t-s)A} \zeta^\beta v_t(0, s)\|_{\leb \infty} \ds \\
    &\sfed I_1(t) + I_2(t) + I_3(t)
    \qquad \text{for $t \in (0, T)$}.
  \end{align*}

  Next, we estimate the terms $I_1$--$I_3$ therein.
  Starting with the first one,
  we apply \eqref{eq:nabla_v_pw:c_1}, \eqref{eq:nabla_v_pw:zeta}, \eqref{eq:nabla_v_pw:v_w1p_pw} \eqref{eq:cond_v0_p}
  and \eqref{eq:nabla_v_pw:def_kappa} to obtain
  \begin{align}\label{eq:nabla_v_pw:i1}
          I_1(t)
    &\le  \lc{sg1} \ure^{-t} \|(\zeta^\beta \tilde v(\cdot, 0))_r\|_{\leb\infty} \notag \\
    &\le  \lc{sg1} \left(
            \|\zeta^\beta v_{0r}\|_{\leb\infty}
            + \beta \|\zeta^{\beta-1} \zeta_r \tilde v(\cdot, 0)\|_{\leb\infty}
          \right) \notag \\
    &\le  \lc{sg1} \left(
            \lc{zeta} \|r^\beta v_{0r}\|_{\leb\infty}
            + \lc{v_pw} \lc{zeta}^{|\beta-1|} \lc{zeta_r} \beta \|r^{\beta-1+\kappa}\|_{\leb\infty}
          \right) \notag \\
    &\le  \lc{sg1} \left(
            \lc{zeta} M
            + \lc{v_pw} \lc{zeta}^{|\beta-1|} \lc{zeta_r} \beta R^{\beta+\kappa-1}
          \right)
    \qquad \text{for $t \in (0, T)$}.
  \end{align}

  By \eqref{eq:nabla_v_pw:c_2}, we moreover have 
  \begin{align}\label{eq:nabla_v_pw:i2}
          I_2(t)
    &\le  \lc{sg2} \int_0^t (t-s)^{\gamma_1} \ure^{-(t-s) \delta_1} \|b_1 \tilde v(\cdot, s) + b_2 v_r(\cdot, s) + b_3 g(\cdot, s)\|_{\leb[(0, R)]p} \ds
    \qquad \text{for $t \in (0, T)$}.
  \end{align}
  Therein are
  \begin{align}
          \|b_1 \tilde v(\cdot, s)\|_{\leb{p}}^{p}
    &\le  \lc{b}^{p} \int_0^R r^{(\beta - 2)p} (\tilde v)^{p}(r, s) \dr
     \le  \lc{v_pw}^{p} \lc{b}^{p} \int_0^R r^{\lambda_1} \dr
     =    \lc{v_pw}^{p} \lc{b}^{p} \frac{R^{\lambda_1+1}}{\lambda_1+1}
     \lt  \infty, \label{eq:nabla_v_pw:b_1} \\
          \|b_2 v_r(\cdot, s)\|_{\leb{p}}^{p} 
    &\le  \lc{b}^{p} \int_0^R \left( r^{n-1} |v_r(r, s)|^{\bar p} \right)^\frac{p}{\bar p} r^\frac{[(\beta-1) \bar p - (n-1)] p}{\bar p} \dr \notag \\
    &\le  \frac{\lc{b}^{p} \|\nabla v(\cdot, s)\|_{\leb[\Omega]{\bar p}}^{p}}{\omega_{n-1}}
            \left( \int_0^R r^{\lambda_2} \dr \right)^\frac{\bar p-p}{\bar p}
     \le  \frac{\lc{v_w1p}^{p} \lc{b}^{p}}{\omega_{n-1}} \left( \frac{R^{\lambda_2+1}}{\lambda_2+1} \right)^\frac{\bar p-p}{\bar p}
     \lt  \infty \label{eq:nabla_v_pw:b_2}
    \intertext{and}
          \|b_3 g(\cdot, s)\|_{\leb{p}}^{p}
    &\le  \lc{b}^{p} \int_0^R \left( r^{n-1} g^\qu(r, s) \right)^\frac{p}{\qu} r^{\frac{[\beta \qu - (n-1)]p}{\qu}} \dr \notag \\
    &\le  \frac{\lc{b}^{p} \|g(\cdot, s)\|_{\leb[\Omega]\qu}^{p}}{\omega_{n-1}} \left( \int_0^R r^{\lambda_3} \dr \right)^\frac{\qu-p}{\qu}
     \le  \frac{M^p \lc{b}^{p}}{\omega_{n-1}} \left( \frac{R^{\lambda_3+1}}{\lambda_3+1} \right)^\frac{\qu-p}{\qu}
     \lt  \infty \label{eq:nabla_v_pw:b_3}
  \end{align}
  for all $s \in (0, T)$ by \eqref{eq:z_eq:b_est}, \eqref{eq:nabla_v_pw:v_w1p_pw}, \eqref{eq:cond_g_p} and \eqref{eq:nabla_v_pw:lambda}.
  Combining \eqref{eq:nabla_v_pw:i2} with \eqref{eq:nabla_v_pw:second_int} and \eqref{eq:nabla_v_pw:b_1}--\eqref{eq:nabla_v_pw:b_3} yields then
  \begin{align}\label{eq:nabla_v_pw:i2_final}
        I_2(t)
    \le \lc{sg2} \lc{second_int} \lc{b} \left(
          \lc{v_pw} \left( \frac{R^{\lambda_1+1}}{\lambda_1+1} \right)^\frac1p
          + \frac{\lc{v_w1p}}{\sqrt[p]{\omega_{n-1}}} \left( \frac{R^{\lambda_2+1}}{\lambda_2+1} \right)^\frac{\bar p-p}{p \bar p}
          + \frac{M}{\sqrt[p]{\omega_{n-1}}} \left( \frac{R^{\lambda_3+1}}{\lambda_3+1} \right)^\frac{\qu-p}{p \qu}
        \right)
    \qquad \text{for all $t \in (0, T)$}.
  \end{align}

  Moreover, 
  as $[\sign(\qu - \frac n2)]_+ = 0$ for $\qu \le \frac n2$,
  for estimating $I_3$ we may assume $\qu \gt \frac n2$ (and hence make use of \eqref{eq:nabla_v_pw:sg3}--\eqref{eq:nabla_v_pw:last_int}).
  Using linearity of $\ure^{\tau A}$ for $\tau \gt 0$,
  integrating by parts and applying \eqref{eq:nabla_v_pw:v_hoelder_time}, \eqref{eq:nabla_v_pw:sg3} and \eqref{eq:nabla_v_pw:last_int},
  we then obtain
  \begin{align}\label{eq:nabla_v_pw:i3}
          I_3(t)
    &=    \left\| \int_0^t \partial_r \ure^{-(t-s)A} \left( \zeta^\beta \partial_s v(0, s) \right) \ds \right\|_{\leb\infty} \notag \\
    &=    \left\| \int_0^t \partial_s [v(0, s) - v(0, t)] \partial_r \ure^{-(t-s)A} \zeta^\beta \ds \right\|_{\leb\infty} \notag \\
    &\le  \left\| \int_0^t [v(0, s) - v(0, t)] \partial_r \partial_s \ure^{-(t-s)A} \zeta^\beta \ds \right\|_{\leb\infty}
          + \left\| \left[ [v(0, s) - v(0, t)] \partial_r \ure^{-(t-s)A} \zeta^\beta \right]_{s=0}^{s=t} \right\|_{\leb\infty} \notag \\
    &\le  \lc{v_hoelder_time} \int_0^t (t-s)^\theta \left\| \partial_r A \ure^{-(t-s)A} \zeta^\beta \right\|_{\leb\infty} \ds 
          + \lc{v_hoelder_time} t^\theta \left\|  \partial_r \ure^{-tA} \zeta^\beta \right\|_{\leb\infty} \notag \\
    &\le  \lc{sg3} \lc{v_hoelder_time} \left(
            \int_0^t s^{\theta + \frac{\beta}{2} - \frac32 - \eps} \ure^{-\delta_2 s} \ds 
            + t^{\theta + \frac{\beta}{2} - \frac12 - \eps} \ure^{-\delta_2 t}
          \right) \|\zeta^\beta\|_{\con\beta} \notag \\
    &\le  \lc{sg3} \lc{v_hoelder_time} \lc{last_int} \lc{zeta_hoelder} 
    \qquad \text{for all $t \in (0, T)$}.
  \end{align}

  Combining \eqref{eq:nabla_v_pw:i1}, \eqref{eq:nabla_v_pw:i2_final} and \eqref{eq:nabla_v_pw:i3}
  shows that $\|z\|_{L^\infty((0, R) \times (0, T))} \le \newlc{z_infty}$
  for some $\lc{z_infty} \gt 0$ only depending on $\Omega$, $M$, $\qu$, $\beta$ and $p_0$.
  Thus, due to the definitions of $\tilde v$ and $z$, \eqref{eq:nabla_v_pw:v_w1p_pw}, \eqref{eq:nabla_v_pw:zeta} and \eqref{eq:nabla_v_pw:def_kappa},
  \begin{align*}
          |v_r(r, t)|
    &=    |\tilde v_r(r, t)| \\
    &=    |\zeta^{-\beta}(r) z_r(r, t) - \beta \zeta^{-\beta-1}(r) \zeta_r(r) z(r, t)| \\
    &\le  \zeta^{-\beta}(r) |z_r(r, t)| + \beta \zeta^{-1}(r) |\zeta_r(r)| |\tilde v(r, t)| \\
    &\le  \lc{zeta}^\beta \lc{z_infty} r^{-\beta} + \lc{v_pw} \lc{zeta} \lc{zeta_r} \beta r^{-1 + \kappa} \\
    &\le  \left( \lc{zeta}^\beta \lc{z_infty} + \lc{v_pw} \lc{zeta} \lc{zeta_r} \beta R^{\beta + \kappa - 1} \right) r^{-\beta}
    \qquad \text{holds for all $(r, t) \in (0, R) \times (0, T)$},
  \end{align*}
  so that we finally arrive at \eqref{eq:main_p:statement}.
\end{proof}

\section{Proofs of the main theorems}
Finally, let us prove Proposition~\ref{prop:main_e}, Theorem~\ref{th:main_p} and Theorem \ref{th:pw_ks}.
\begin{proof}[Proof of Proposition~\ref{prop:main_e} and Theorem~\ref{th:main_p}]
  The corresponding statements have been shown
  in Lem\-ma~\ref{lm:ell_vr_est} and Lemma~\ref{lm:nabla_v_pw}.
\end{proof}

\begin{proof}[Proof of Theorem~\ref{th:pw_ks}]
  For $\pu = 1$, this has already been shown in \cite[Theorem~1.3]{FuestBlowupProfilesQuasilinear2020}.
  Moreover, in the case of $\pu \gt 1$,
  we set $\qu \defs \frac{\pu}{s}$ as well as $g(x, t) \defs f(u(x, t), v(x, t))$ for $x \in \Omega$ and $t \in (0, T)$
  and, for $\alpha \gt \frac{n (ns - \pu)}{[(m - q) n + \pu] \pu} = \frac{\frac{n - \qu}{\qu}}{m - q + \frac{\pu}{n}}$,
  we choose $\tilde \beta \gt \frac{n - \qu}{\qu} = \frac{ns - \pu}{\pu}$ as well as $\theta \gt n$
  such that $\alpha \ge \frac{\tilde \beta}{(m - q) + \frac{\pu}{n} - \frac{\pu}{\theta}}$
  and $m-q \in (\frac{p}{\theta} - \frac{p}{n}, \frac{p}{\theta} + \frac{\tilde \beta \pu - \pu}{n}]$.
  Since we may without loss generality assume $\beta \le \tilde \beta$,
  the statement follows immediately from Theorem~\ref{th:main_p} and \cite[Theorem~1.1]{FuestBlowupProfilesQuasilinear2020}
\end{proof}

\section*{Acknowledgments}
The author is partially supported by the German Academic Scholarship Foundation
and by the Deutsche Forschungsgemeinschaft within the project \emph{Emergence of structures and advantages in
cross-diffusion systems}, project number 411007140.

\footnotesize

\end{document}